\documentclass{amsart}
\usepackage{graphicx}
\usepackage{amssymb}
\usepackage{amsfonts}
\swapnumbers
\newtheorem{theorem}{Theorem}[section]
\newtheorem{lemma}[theorem]{Lemma}

\theoremstyle{definition}

\numberwithin{equation}{section}
 \theoremstyle{plain}
 
 \numberwithin{equation}{section} 
 \numberwithin{figure}{section} 
 \theoremstyle{plain}
 \theoremstyle{remark}
 \newtheorem*{acknowledgement*}{Acknowledgement}

\newcommand{\cF}{{\mathcal F}}
\newcommand{\cG}{{\mathcal G}}
\newcommand{\cH}{{\mathcal H}}
\newcommand{\cI}{{\mathcal I}}
\newcommand{\cJ}{{\mathcal J}}

\newcommand{\cL}{{\mathcal L}}

\newcommand{\cN}{{\mathcal N}}

\newcommand{\cQ}{{\mathcal Q}}

\newcommand{\cT}{{\mathcal T}}

\newcommand{\Te}{{\Theta}}

\newcommand{\vt}{{\vartheta}}
\newcommand{\Om}{{\Omega}}

\newcommand{\ve}{{\varepsilon}}
\newcommand{\del}{{\delta}}
\newcommand{\Del}{{\Delta}}

\newcommand{\Gam}{{\Gamma}}

\newcommand{\vr}{{\varrho}}

\newcommand{\sig}{{\sigma}}

\newcommand{\La}{{\Lambda}}

\newcommand{\up}{{\upsilon}}

\newcommand{\bbR}{{\mathbb R}}

\newcommand{\bbI}{{\mathbb I}}

\begin{document}
\title[]{Error estimates for discrete approximations of game options\\
with multivariate diffusion asset prices}%
 \vskip 0.1cm
 \author{ Yuri Kifer\\
\vskip 0.1cm
 Institute  of Mathematics\\
Hebrew University\\
Jerusalem, Israel}%
\address{
Institute of Mathematics, The Hebrew University, Jerusalem 91904, Israel}
\email{ kifer@math.huji.ac.il}%

\thanks{ }
\dedicatory{To the memory of Hiroshi Kunita  }
 \date{\today}
\begin{abstract}\noindent
We obtain error estimates for strong approximations of a diffusion with a diffusion matrix $\sig$ and a drift $b$ by the discrete time process defined recursively
\[
X_N((n+1)/N)=X_N(n/N)+N^{-1/2}\sig(X_N(n/N))\xi(n+1)+N^{-1}b(X_N(n/N)),
\]
where $\xi(n),\, n\geq 1$ are i.i.d. random vectors, and apply this in order to approximate the fair
 price of a game option with a diffusion asset price evolution by values of Dynkin's games with payoffs
 based on the above discrete time processes. This provides an
  effective tool for computations of fair prices of game options with path dependent payoffs in
  a multi asset market with diffusion evolution.
\end{abstract}
\maketitle
\markboth{Yu.Kifer}{Error estimates}
\renewcommand{\theequation}{\arabic{section}.\arabic{equation}}
\pagenumbering{arabic}

\section{Introduction}\label{sec1}\setcounter{equation}{0}

In the present paper we continue the line of research in \cite{Ki06} and \cite{Ki07} approximating
game options whose stocks evolutions are described by multidimensional diffusion processes. This is
done constructing first strong approximations of the diffusion by a sequence of discrete time processes
estimating $L^2$ errors of these approximations. These processes are discrete time processes obtained
recursively for $n=0,1,...,N-1$ by
\[
X_N((n+1)/N)=X_N(n/N)+N^{-1/2}\sig(X_N(n/N))\xi(n+1)+N^{-1}b(X_N(n/N)),
\]
where $X_N(0)=x_0$, $\sig$ and $b$ is a matrix and a vector functions, respectively and
$\xi(n),\, n\geq 1$ is a sequence of i.i.d. random vectors with $E\xi(1)=0$. The strong approximation
method enables us to redefine both the sequence $\xi(n),\, n\geq 1$ and the limiting diffusion
$d\Xi(t)=\sig(\Xi(t))dW(t)+b(\Xi(t))dt$ preserving their distributions on a same sufficiently rich
probability space so that the $L^2$-distance between them have the order $N^{-\del}$ for some $\del>0$.
In the second step we compare fair prices of options with payoffs
based on these discrete approximations with the fair price of the option with the diffusion asset
evolution. This is not straightforward since this prices are given by values of the corresponding
Dynkin games which depend on sets of stopping times involved and the latter are different for the
approximations and for the limiting diffusion.
The payoffs of the above game options are supposed
to be path dependent, and so free boundary partial differential equations methods cannot help here and
discrete time approximations is the only possible approach in this situation to fair prace computations
taking into account that in the discrete time case we can employ the dynamical programming (backward
recursion) algorithm.

The setup of this paper is the special case of a more general discrete time setup in our recent paper
\cite{Ki20+} but the latter paper provides a detailed proof mostly in the continuous time averaging setup while here we deal with the more specific discrete time setup which can be described in the more
transparent way. The motivation both for \cite{Ki20+} and for the present paper comes, in particular,
from the series of papers \cite{Kha}, \cite{Bor}, \cite{BF} and \cite{CE} on the weak diffusion limit
in averaging and from the series of papers \cite{BP}, \cite{KP}, \cite{MP1} and \cite{DP}  on strong approximations (see \cite{Ki20+} for more detailes). The former papers yielded only weak convergence
 results while the latter dealt only with approximations of the Brownian motion. We observe that in
 the one dimensional case it is still possible to use an extended version of the Skorokhod embedding
 into martingales which also yields error estimates for approximations (see \cite{BDG}).
 Approximation similar to ours appeared previously in \cite{He} but only the weak convergence to a
 diffusion was established there which, in principle, could not provide any error estimates. The reader
 may compare our approach with the well known Euler--Maruyama approximation of solutions of stochastic
 differential equations (see, for instance, \cite{Mao}) where $\xi(n)$'s are increments of the Brownian
 motion. In our setup $\xi(n)$'s are quite general and, in particular, we can take i.i.d. random vectors
 taking on only few values which can be useful in applications since they are easier to simulate and compute
 than Gaussian random vectors.

 \section{Preliminaries and main results}\label{sec2}\setcounter{equation}{0}

We start with a complete probability space $(\Om,\,\cF,\, P)$, a sequence of independent identically
distributed (i.i.d.) random vectors $\xi(n),\, n\geq 1$ and a diffusion process $\Xi$ solving the
stochastic differential equation
\begin{equation}\label{2.1}
d\Xi(t)=\sig(\Xi(t))dW(t)+b(\Xi(t))dt
\end{equation}
where $W$ is the $d$-dimensional continuous Brownian motion while $\sig$ and $b$ are bounded
Lipschitz continuous
$d\times d$ matrix and $d$-dimensional vector functions, respectively. Namely, we assume that for
some constant $L\geq 1$ and all $x,y\in\bbR^d$,
\begin{equation}\label{2.2}
|\sig(x)|\leq L,\, |b(x)|\leq L,\, |\sig(x)-\sig(y)|\leq L|x-y|,\, |b(x)-b(y)|\leq L|x-y|
\end{equation}
where $|\cdot |$ denotes the Euclidean norm of a vector or of a matrix. We assume also that
\begin{equation}\label{2.3}
E\xi(1)=0,\,E(\xi_i(1)\xi_j(1))=\del_{ij}\,\,\mbox{and}\,\, |\xi(1)|\leq L\,\,\mbox{almost surely
 (a.s.)}
 \end{equation}
 where $\xi(n)=(\xi_1(n),...,\xi_d(n))$ and $\del_{ij}$ is the Kronecker delta. Next, we consider
 the sequence of discrete time processes $X_N,\, N\geq 1$ on $\bbR^d$ defined recursively for
  $n=0,1,...,N-1$ by
 \begin{equation}\label{2.4}
X_N((n+1)/N)=X_N(n/N)+N^{-1/2}\sig(X_N(n/N))\xi(n+1)+N^{-1}b(X_N(n/N))
\end{equation}
where $X_N(0)=\Xi(0)=x_0$ is fixed. We extend $X_N$ to the continuous time setting
\begin{equation}\label{2.5}
X_N(t)=X_N(n/N)\quad\mbox{if}\quad n/N\leq t<(n+1)/N
\end{equation}
and without loss of generality we will assume that all our processes evolve on the time interval
$[0,1]$ so that $n$ runs in (\ref{2.5}) from 0 to $N$. The following result will be proved in
Sections \ref{sec3} and \ref{sec4}.

\begin{theorem}\label{thm2.1} Suppose that the conditions  (\ref{2.2}) and
(\ref{2.3}) hold true and that the probability space $(\Om,\,\cF,\, P)$ is rich enough so that
there exist a sequence of i.i.d. uniformly distributed random variables defined on it. Then for
each integer $N\geq N_0=((10^8d)^{24d}+1)^4$ there exists a $d$-dimensional Brownian motion
$W=W_N$ such that the strong
solution $\Xi=\Xi_N$ of the stochastic differential equation (\ref{2.1}) with such $W$ and the
initial condition $X_N(0)=\Xi(0)=x_0$ satisfies
\begin{equation}\label{2.6}
E\sup_{0\leq t\leq 1}|X_N(t)-\Xi(t)|^{2}\leq C_0[N^{\frac 14}]^{-\frac 1{50d}}
\end{equation}
where $C_0=C_3e^{C_4}+2L^2(L^2+1)+40L^2$ with $C_3$ and $C_4$ defined at the end of Section \ref{sec4}.   In particular, the Prokhorov distance between the path distributions of $X_N$ and of $\Xi$ is
bounded by $C^{1/3}_0[N^{\frac 14}]^{-\frac 1{150d}}$.
\end{theorem}
We observe that though we may have to redefine the Brownian motion $W$ for each $N$ separately the path
distribution of the diffusion $\Xi$ remains the same since it is continuous and the coefficients of the
 stochastic differential equation (\ref{2.1}) do not change, and so we have all the time the same
 Kolmogorov equation and the same martingale problem (see \cite{SV}). Clearly, the estimate (\ref{2.6})
 is meaningful only for large $N$ and we provide it for all $N\geq N_0$ though, of course, an explicit
 estimate in (\ref{2.6}) can be obtained also when $1\leq N\leq N_0$ taking into account that then
 $|X_N(t)|\leq L(\sqrt {N_0}L+1)$ and $E\sup_{0\leq t\leq 1}|\Xi(t)|^{2}\leq 5L^2$. Theorem \ref{thm2.1}
 can also be derived under some moment boundedness conditions rather than the uniform bounds in (\ref{2.2})
 which are assumed to reduce technicalities in our exposition. Observe also that if we
 consider time dependent coefficients $\sig(t,x)$ and $b(t,x)$ Lipschitz continuous in both variables and
 take $\sig(n/N,X_N(n/N))$ and $b(n/N,X_N(n/N))$ in (\ref{2.4}) in place of $\sig(X_N(n/N))$ and $b(X_N(n/N))$,
 then we will obtain $X_N(t)$ which approximates a time inhomogeneous diffusion with coefficients $\sig(t,x)$
 and $b(t,x)$ having essentially the same error estimates as in (\ref{2.6}).

  Next, we will describe an application of our results to computations of values of Dynkin's optimal
   stopping games and fair prices of game options with the payoff function having the form
  \begin{equation}\label{2.7}
  R^\Xi(s,t)=G_s(\Xi)\bbI_{s<t}+F_t(\Xi)\bbI_{t\leq s}
  \end{equation}
  where $\Xi$ is a diffusion solving the stochastic differential equation (\ref{2.1}). Here, $G_t\geq
  F_t$ and both are functionals on paths for the time interval $[0,t]$ satisfying
  certain regularity conditions specified below. Thus, if the first player stops at the time $s$ and
  the second one at the time $t$ then the former pays to the latter the amount $R^\Xi(s,t)$. The game
  runs until the termination time 1 when the game stops automatically, if it was not stopped
  before, and then the first player pays to the second one the amount $G_1(\Xi)=F_1(\Xi)$. Clearly,
  the first player tries to minimize the payment while the second one tries to maximize it. Under the
  conditions below this game has the value (see, for instance, Section 6.2.2 in \cite{Ki20}),
  \begin{equation}\label{2.8}
  V^{\Xi}=\inf_{\sig\in\cT_{01}^\Xi}\sup_{\tau\in\cT_{01}^{\Xi}}ER^{\Xi}(\sig,\tau)
  \end{equation}
  where $\cT^\Xi_{01}$ is the set of all stopping times $0\leq\tau\leq 1$  with respect to the
  filtration $\cF_t^\Xi,\, t\geq 0$ generated by the diffusion $\Xi$ or, which is the same, generated
  by the Brownian motion $W$.

  When we are talking about asset prices then usually
  it is assumed that they are nonnegative, and so a diffusion with bounded coefficients maybe not a
  good model for a description of evolution of these prices. It maybe more appropriate to assume that
  the asset prices evolve according to the vector process described by exponents $(\exp(\Xi^{(i)}(t)),
  \, i=1,...,d)$ where $\Xi^{(1)},...,\Xi^{(d)}$ are components of the vector $\Xi$. Nevertheless, it will
  be more convenient for us to speak about the diffusion $\Xi$ itself and to impose conditions on the
  payoff functionals $F$ and $G$ such that exponential functionals will be allowed which will amount to
  the same effect as exponents describing the evolution of asset prices. Another important point that
  the fair price of a game option equals the value of the corresponding Dynkin optimal stopping game
  considered with respect to the equivalent martingale measure, i.e. with respect to the probability for which the assets evolution is described by a martingale (see \cite{Ki20}) provided that the
   interest rate is supposed to be zero. When the asset prices are given by the above exponential
   formula the probability $P$ itself will be a martingale measure provided $b_i(x)=
   -\frac 12\sum_{j=1}^d\sig^2_{ij}(x)$, $i=1,...,d$ for each $x\in\bbR^d$. Otherwise, we have to
   perform all estimates with respect to a martingale measure $Q$ and, in particular, $\xi(1),\xi(2),...$ in (\ref{2.4}) should be an i.i.d. sequence satisfying (\ref{2.3}) with respect
   to the probability $Q$. According to the Girsanov theorem (see, for instance, Section 7.4.3
   in \cite{Ki20}),
   \[
   \frac {dQ}{dP}=\La\,\,\mbox{where}\,\,\La=\exp(-\int_0^1\langle\zeta(s),dW(s)\rangle-\frac 12
   \int_0^1|\zeta(s)|^2ds)
   \]
   where $\langle\cdot,\cdot\rangle$ is the inner product and the vector process $\zeta(s)$ satisfies
   \[
   \sig(\Xi(s))\zeta(s)=b(\Xi(s))+\frac 12\eta(\Xi(s))\,\,\mbox{and}\,\,\eta=(\eta_1,...,\eta_d),\,
   \eta_i(x)=\sum_{i=1}^d\sig^2_{ij}(x).
   \]
   Since our estimates do not depend explicitly on the probability measure once the setup above is
   preserved and since there is no one preferable stock evolution model here, we will not discuss this
   point further, and so, strictly speaking, we will deal with the approximation of the Dynkin game value $V^{\Xi}$ and not of the fair price of the corresponding game option, i.e. we will make estimates with respect to the probability $P$ and not with respect to an equivalent martingale measure which depends on a choice of the stock evolution model.

    We assume that $F_t$ and $G_t,\, t\in[0,1]$ are continuous functionals on the space $M_d[0,t]$ of
  bounded Borel measurable maps from $[0,t]$ to $\bbR^d$  considered with the uniform metric
   $d_{0t}(\up,\tilde\up)=\sup_{0\leq s\leq t}|\up_s-\tilde\up_s|$ and there exists a constant $K>0$
   such that
    \begin{eqnarray}\label{2.9}
  & |F_t(\up)-F_t(\tilde\up)|+|G_t(\up)-G_t(\tilde\up)|\\
  &\leq K(d_{0t}(\up,\tilde\up)+\bbI_{\sup_{0\leq
    u\leq t}|\up_u-\tilde\up_s|>1})\exp(K\sup_{0\leq u\leq t}(|\up_u|+|\tilde\up_u|))\nonumber
   \end{eqnarray}
   and
   \begin{equation}\label{2.10}
   |F_t(\up)-F_s(\up)|+|G_t(\up)-G_s(\up)|\leq K(|t-s|+\sup_{u\in[s,t]}|\up_u-\up_s|)\exp(K\sup_{0\leq
    u\leq t}|\up_u|).
   \end{equation}

 Next, we will consider Dynkin's games with payoffs based on the process $X_N$,
  \begin{equation}\label{2.11}
  R_N(s,t)=G_s(X_N)\bbI_{s<t}+F_t(X_N)\bbI_{t\leq s}.
  \end{equation}
 Denote by $\cF^\xi_{mn},\, m\leq n$ the $\sig$-algebra generated by $\xi(m),...,\xi(n)$ and let
 $\cT^\xi_{mn}$ be the set of all stopping times with respect to the filtration $\cF^\xi_{0k},\,
  k\geq 0$ taking on values
 $m,m+1,...,n$. We allow also any stopping time to take on the value $\infty$, i.e. we allow players
 not to stop the game at all, but anyway the game is stopped automatically at the termination time 1
  and then the first player pays to the second one the amount $G_1(X_N)=F_1(X_N)$.
 Now the game value of the Dynkin game in this setup is given by
 \begin{equation}\label{2.12}
 V_N=\inf_{\zeta\in\cT^\xi_{0N}}\sup_{\eta\in\cT^\xi_{0N}}ER_N(\zeta/N,\,\eta/N).
 \end{equation}

 \begin{theorem}\label{thm2.2} Suppose that the conditions (\ref{2.9}) and (\ref{2.10}) as well as
  the conditions of Theorem \ref{thm2.1} hold true. Then for each $\del>0$ there exists $C_\del>0$
   such that for any integer $N\geq N_0$,
 \begin{equation}\label{2.13}
 |V^\Xi-V_N|\leq C_\del[N^{\frac 14}]^{\del-\frac 1{100d}}
 \end{equation}
 where $C_\del$ does not depend on $N$ and for each $\del$ it can be estimated explicitly from
 the proof in Section \ref{sec5}.
 \end{theorem}

 Since we use in Theorem \ref{thm2.2} a specific construction of the diffusion $\Xi$ from Theorem
 \ref{thm2.1} it is important to note that the game value $V^\Xi$ depends only on the path distribution
 of $\Xi$, i.e. only on the diffusion coefficients $\sig$ and $b$, and not on a choice of the Brownian
 motion in the stochastic differential equation (\ref{2.1}) (see \cite{Dol}).
 We observe also that the main advantage in computation $V_N$ in comparison to $V^\Xi$ is the
 possibility to use the dynamical programming (backward recursion) algorithm. Namely, set
 $V_{NN}=F_{1}(X_N)$ and recursively for $n=N-1,...,1,0$,
 \begin{equation}\label{2.14}
 V_{Nn}=\min\big(G_{n/N}(X_N),\,\max(F_{n/N}(X_N),\, E(V_{N,n+1}|\cF^\xi_{0,n}))\big).
 \end{equation}
 Then $V_{N0}=V_N$ (see, for instance, Section 6.2.2 in \cite{Ki20}). Of course, the computation of conditional expectations above becomes complicated if the $\sig$-algebras $\cF^\xi_{0n}$ are big
 but if we choose independent random vectors $\xi(n)$ in (\ref{2.4}) taking on only few values then these
  $\sig$-algebras contain not so many sets and the conditional expectations can be computed easily.
  Observe also that in the particular case when the diffusion $\Xi$ is just a multidimensional Brownian
   motion, a result similar to Theorem \ref{thm2.2} was obtained in \cite{Ki07} where it was
   sufficient to consider the standard normalized sums of random vectors $\xi(n)$ rather than the
   more subtle case of difference equations (\ref{2.4}).

   \section{Auxiliary estimates }\label{sec3}\setcounter{equation}{0}
Set $n_k=k[N^{\frac 14}],\, k=0,1,...,k_N$ where $k_N=[N/[N^{\frac 14}]]$ where $[\cdot]$ denotes
the integral part. Define
\begin{eqnarray}\label{3.1}
&\hat X_N(t)=x_0+N^{-1/2}\sum_{0\leq k\leq k_N(t)}\big(\sig(X_N(\frac {n_k}N))\sum_{n_k<l\leq
n_{k+1}\wedge[Nt]}\xi(l)\\
&+N^{-1/2}b(X_N(\frac {n_k}N))(n_{k+1}\wedge [Nt]-n_k)\big)\nonumber
\end{eqnarray}
where $k_N(t)=\max\{ k:\, n_k\leq Nt\}$.

\begin{lemma}\label{lem3.1} For any $N\geq 1$,
\begin{equation}\label{3.2}
E\sup_{0\leq t\leq 1}|X_N(t)-\hat X_N(t)|^2\leq 136L^8N^{-1/2}.
\end{equation}
\end{lemma}
\begin{proof}
First, we write
\begin{equation}\label{3.3}
|X_N(t)-\hat X_N(t)|^2\leq 2|M(t)|^2+2|J(t)|^2
\end{equation}
where
\[
M(t)=N^{-1/2}\sum_{0\leq k\leq k_N(t)}\sum_{n_k<l\leq
n_{k+1}\wedge[Nt]}\big(\sig(X_N(\frac {l}N))-\sig(X_N(\frac {n_k}N))\big)\xi(l+1)
\]
and
\[
J(t)=N^{-1}\sum_{0\leq k\leq k_N(t)}\sum_{n_k<l\leq
n_{k+1}\wedge[Nt]}(b(X_N(\frac {l}N))-b(X_N(\frac {n_k}N))).
\]

Recall that if $h=h(x,y)$ is a bounded Borel function, $\cG\subset\cF$ is a $\sig$-algebra and
$Y,Z$ are random variables such that $Y$ is $\cG$-measurable and $Z$ is independent of $\cG$, then
$E(h(Y,Z)|\cG)=g(Y)$ where $g(x)=Eh(x,Z)$. It follows from here and from (\ref{2.3}) that $M(t),\,
0\leq t\leq 1$ is a martingale. Hence, by (\ref{2.3}) and the Doob martingale inequality (see,
for instance, Section 6.1.2 in \cite{Ki20}),
\begin{eqnarray}\label{3.4}
&E\sup_{0\leq t\leq 1}|M(t)|^2\leq 4E|M(1)|^2\\
&=4N^{-1}\sum_{0\leq k\leq k_N(t)}\sum_{n_k<l\leq n_{k+1}\wedge[Nt]}E|\big(\sig(X_N(\frac {l}N))-\sig(X_N(\frac {n_k}N))\big)\xi(l+1)|^2.\nonumber
\end{eqnarray}
By (\ref{2.2})--(\ref{2.4}) for $n_k<l\leq n_{k+1}$,
\begin{eqnarray}\label{3.5}
&E|\big(\sig(X_N(\frac {l}N))-\sig(X_N(\frac {n_k}N))\big)\xi(l+1)|^2
\leq L^4E|X_N(\frac {l}N)-X_N(\frac {n_k}N)|^2\\
&\leq 2L^4\big(N^{-1}E|\sum_{n_k\leq m<l}
\sig(X_N(m/N))\xi(m+1)|^2\nonumber\\
&+N^{-2}E|\sum_{n_k\leq m<l}b(X_N(m/N))\xi(m+1)|^2\big)\leq 16L^8N^{-1/2}.\nonumber
\end{eqnarray}

Now, by (\ref{2.2}) and (\ref{2.3}),
\[
E\sup_{0\leq t\leq 1}|J(t)|^2\leq\sum_{0\leq k\leq k_N}\sum_{n_k<l\leq n_{k+1}\wedge N}E|X_N(\frac {l}N)-X_N(\frac {n_k}N)|^2
\]
and for $n_k<l\leq n_{k+1}$,
\[
|X_N(\frac {l}N)-X_N(\frac {n_k}N)|\leq N^{-1/2}(L^2+N^{-1/2}L)[N^{\frac 14}]\leq
2L^2N^{-\frac 14}.
\]
These together with (\ref{3.3})--(\ref{3.5}) yield (\ref{3.2}).
\end{proof}

Next, we estimate the characteristic function of a sum of independent random vectors which is well
known but for completeness and in order to provide explicit constants we provide the details.
\begin{lemma}\label{lem3.2} For any integer $n \geq 1$ and $x\in\bbR^d$,
 \begin{equation}\label{3.6}
 |f_n(x,w)-\exp(-\frac 12\langle A(x)w,w\rangle)|\leq C_1n^{-\wp}
 \end{equation}
 for all $w\in\bbR^d$ with $|w|\leq n^{\wp/2}$ where $A(x)=\sig(x)\sig^*(x)$,
 \[
 f_n(x,w)=E\exp(i\langle w,\, n^{-1/2}\sig(x)\sum_{0<l\leq n}\xi(l)\rangle),
 \]
 $\wp=1/6$ and $C_1=\frac 32L^6$.
 \end{lemma}
 \begin{proof}
 Set $m_j=j[\sqrt n],\, j=0,1,...,m(n)$, $m(n)=\max\{ j:\, j[\sqrt n]\leq n\}$, $y_j=
 \sig(x)\sum_{m_j<l\leq m_{j+1}\wedge n}\xi(l)$ and $\eta_j=\langle w,\, n^{-1/2}y_j\rangle$.
 Now we have
 \begin{equation}\label{3.7}
 |f_n(x,w)-\exp(-\frac 12\langle A(x)w,w\rangle)|\leq I_1+I_2
 \end{equation}
 where
  \begin{equation}\label{3.8}
  I_1=|E\exp(i\sum_{0\leq j\leq m(n)}\eta_j)-\prod_{0\leq j\leq m(n)}Ee^{i\eta_j}|=0,
  \end{equation}
  since $\eta_j,\, j=1,...,m(n)+1$ are independent random variables, and
  \begin{eqnarray}\label{3.9}
  &I_2=|\prod_{0\leq j\leq m(n)}Ee^{i\eta_j}-\exp(-\frac 12\langle A(x)w,w\rangle)|\\
  &\leq \sum_{0\leq j\leq m(n)}|Ee^{i\eta_j}-\exp(-\frac {(m_{j+1}\wedge n-m_j)}{2n}\langle A(x)w,w\rangle)|\nonumber
  \end{eqnarray}
  where we use that
  \[
  |\prod_{1\leq j\leq l}a_j-\prod_{1\leq j\leq l}b_j|\leq\sum_{1\leq j\leq l}|a_j-b_j|
  \]
  whenever $0\leq |a_j|, |b_j|\leq 1,\, j=1,...,l$.

  Using (\ref{2.3}) and the inequalities
  \[
 |e^{ia}-1-ia+\frac {a^2}2|\leq |a|^3\,\,\mbox{and}\,\, |e^{-a}-1+a|\leq a^2\,\,\mbox{if}\,\, a\geq 0,
 \]
 we obtain that
 \begin{eqnarray}\label{3.10}
 &|Ee^{i\eta_k}-\exp(-\frac {(m_{j+1}\wedge n-m_j)}{2n}\langle A(x)w,w\rangle)|\\
 &\leq\frac 12|E\eta^2_j-\frac {(m_{j+1}\wedge n-m_j)}{n}\langle A(x)w,w\rangle|+E|\eta_j|^3+
 \frac 1{4n}|\langle A(x)w,w\rangle|^2.\nonumber
 \end{eqnarray}
 Now, by (\ref{2.3}) and the independency of $\xi(l)$'s,
 \begin{equation}\label{3.11}
 E\eta_j^2=n^{-1}\sum_{m_j<l\leq m_{j+1}\wedge n}E\langle w,\sig(x)\xi(l)\rangle^2=n^{-1}
 (m_{j+1}\wedge n-m_j)\langle A(x)w,w\rangle.
 \end{equation}
 Hence,
 \[
 I_2\leq(\sqrt n+1)(n^{-3/2}L^6|w|^6+\frac 14n^{-1}L^4|w|^4)
 \]
 and (\ref{3.6}) follows.
  \end{proof}

  Set $Y_{N,k}(x)=\sig(x)\sum_{n_k<l\leq n_{k+1}}\xi(l)$ for $k=0,1,...,k_N-1$ and $Y_{N,k_N}(x)=
  \sig(x)\sum_{n_{k_N}<l\leq N}\xi(l)$. As a corollary of Lemma \ref{lem3.2} we obtain
  \begin{lemma}\label{lem3.3} For any integer $N\geq 1$ and $k=0,1,...,k_N-1$,
  \begin{equation}\label{3.12}
  |E\big(\exp(i\langle w,\,(n_{k+1}-n_k)^{-1/2}Y_{N,k}(X_N(\frac {n_k}N))\rangle)|\cF_{0n_k}^\xi\big)-
  g_{X_N(\frac {n_k}N)}(w)|\leq C_1(n_{k+1}-n_k)^{-\wp}
  \end{equation}
  for all $w\in\bbR^d$ with $|w|\leq (n_{k+1}-n_k)^{\wp/2}$,
  where $g_x(w)=\exp(-\frac 12\langle A(x)w,w\rangle)$ and, recall, $\cF_{0n}^\xi=\sig\{\xi(1),...,\xi(n)\}$.
  \end{lemma}
  \begin{proof}
  Since $X_N(\frac {n_k}N)$ is $\cF_{0n_k}^\xi$-measurable and $\sum_{n_k<j\leq n_{k+1}}\xi(l)$
  is independent of $\cF_{0n_k}^\xi$, it follows that
  \[
  E\big(\exp(i\langle w,\,(n_{k+1}-n_k)^{-1/2}Y_{N,k}(X_N(\frac {n_k}N))\rangle|\cF_{0n_k}^\xi\big)=
  f_{n_{k+1}-n_k}(X_N(\frac {n_k}N),w),
  \]
  where $f_n(x,w)$ was defined in Lemma \ref{lem3.2}, and so (\ref{3.12}) follows from (\ref{3.6}).
   \end{proof}

   \section{Strong approximation }\label{sec4}\setcounter{equation}{0}
The strong approximations here will be based on the following result which is a slight variation of
Theorem 3 and Remark 2.6 from \cite{MP1} with the additional feature from Theorem 4.6 of \cite{DP}
that we enrich the probability space by a sequence of i.i.d. uniformly distributed random variables
and not just by one such random variable and this result follows by essentially the same proofs as in
the cited above papers.
\begin{theorem}\label{thm4.1}
Let $\{ V_m,\, m\geq 1\}$ be a sequence of random vectors with values in $\bbR^d$ defined on some
probability space $(\Om,\cF,P)$ and such that $V_m$ is measurable with respect to $\cF_m$, $m=1,2,...$
where $\cF_m,\, m\geq 1$ is a filtration of sub-$\sig$-algebras of $\cF$. Let $\cG_m$ and $\cH_m$,
$m=0,1,...$ be two increasing sequences of countably generated sub-$\sig$-algebras of $\cF$ such that
$\cH_m\subset\cG_m\subset\cF_m$ for each $m\geq 1$. Assume that the probability space is rich enough
so that there exists on it a sequence of uniformly distributed on $[0,1]$ independent random variables
 $U_m,\, m\geq 1$ independent of $\vee_{m\geq 0}\cG_m$. For each $m\geq 1$, let $G_m(\cdot|\cH_{m-1})$
 be a regular conditional distribution on $\bbR^d$, measurable with respect to $\cH_{m-1}$ and with
 the conditional characteristic function
 \[
 g_m(w|\cH_{m-1})=\int_{\bbR^d}\exp(i\langle w,x\rangle)G_m(dx|\cH_{m-1}),\,\, w\in\bbR^d.
 \]
 Suppose that for some non-negative numbers $\nu_m,\del_m$ and $K_m\geq 10^8d$,
 \begin{equation}\label{4.1}
 \int_{|w|\leq K_m}E\big\vert E(\exp(\langle w,V_m\rangle)|\cG_{m-1})-g_m(w|\cH_{m-1})\big\vert dw
 \leq\nu_m(2K_m)^d
 \end{equation}
 and that
 \begin{equation}\label{4.2}
 E\big(G_m(\{ x:\, |x|\geq\frac 12K_m\}|\cH_{m-1})\big)<\del_m.
 \end{equation}
 Then there exists a sequence $\{ W_m,\, m\geq 1\}$ of $\bbR^d$-valued random vectors defined on
 $(\Om,\cF,P)$ with the properties

 (i) $W_m$ is $\cG_m\vee\sig\{U_m\}$-measurable for each $m\geq 1$;

 (ii) $G_m(\cdot |\cH_{m-1})$ is conditional distribution of $W_m$ given $\sig\{ U_1,...,U_{m-1}\}\vee
 \cG_{m-1}$, in particular, $W_m$ is conditionally independent of $\sig\{ U_1,...,U_{m-1}\}\vee
 \cG_{m-1}$ (and so also of $W_1,...,W_{m-1})$ given $\cH_{m-1},\, m\geq 1$;

 (iii) Let $\vr_m=16K^{-1}_m\log K_m+2\nu_m^{1/2}K_m^d+2\del_m^{1/2}$. Then
 \begin{equation}\label{4.3}
 P\{ |V_m-W_m|\geq\vr_m\}\leq\vr_m
 \end{equation}
 and, in particular, the Prokhorov distance between the distributions $\cL(V_m)$ and $\cL(W_m)$
 of $V_m$ and $W_m$, respectively, does not exceed $\vr_m$.
 \end{theorem}

Now, in the notations of Theorem \ref{thm4.1} we set $V_k=(n_k-n_{k-1})^{-1/2}Y_{N,k-1}(X_N(
\frac {n_{k-1}}N))$,  $\cF_k=\cG_k=\cF^\xi_{0n_k}$, $\cH_k=\sig\{ X_N(\frac {n_k}N)\}$ and
$g_k(w|\cH_{k-1})=g_{X_N(\frac {n_{k-1}}N)}(w)$ where $g_x$ was defined in Lemma \ref{lem3.3}.
Thus, $G_k(\cdot|\cH_{k-1})=G_{X_N(\frac {n_{k-1}}N)}(\cdot)$ where $G_x$ is the mean zero $d$-dimensional Gaussian distribution with the covariance matrix $A(x)$ and the characteristic
function $g_x$. By Lemma \ref{lem3.3},
 \begin{eqnarray}\label{4.4}
 &\quad\int_{|w|\leq K_{k}}E\big\vert E\big(\exp(i\langle w,V_{k}\rangle)|
 \cG_{k-1}\big)-g_{k}(w|\cH_{k-1})\big\vert dw\\
 &\leq C_1(n_k-n_{k-1})^{-\wp}(2K_k)^d\leq 2^dC_1[N^{\frac 14}]^{-1/8}\nonumber
 \end{eqnarray}
 where we take $K_k=[N^{\frac 14}]^{\frac 1{24d}}<(n_k-n_{k-1})^{\wp/2}$. Next, for each
 $x\in\bbR^d$ let $\Te_x$ be a mean zero Gaussian random variable with the
covariance matrix $A(x)$. Then by (\ref{2.2}) and the Chebyshev inequality,
\begin{eqnarray}\label{4.5}
&E\big( G_{k}(\{ y\in\bbR^d:\, |y|\geq\frac 12K_k\}|\cH_{k-1})\big)\\
&\leq\sup_{y\in\bbR^d}P\{|\Te_y| \geq\frac 12[N^{\frac 14}]^{\frac 1{24d}}\}\leq 4L^2d
[N^{\frac 14}]^{-\frac 1{12d}}.\nonumber
\end{eqnarray}

In order to use Theorem \ref{thm4.1} we need that $K_k\geq 10^8d$ and this will hold true if
 $N\geq N_0=((10^8d)^{24d}+1)^4$ which is the assumption of Theorem \ref{thm2.1}.
Now, Theorem \ref{thm4.1} provides us with random vectors $\{ W_k,\, k\geq 1\}$ satisfying
the properties
(i)--(iii), in particular, given $X_{N}(\frac {n_{k-1}}N)$, the random vector $W_{k}$ has the mean zero
 Gaussian distribution with the covariance matrix $A(X_{N}(\frac {n_{k-1}}N))$ and it is conditionally
  independent of $\cG_{k-1}$ and of $W_1,...,W_{k-1}$ while in view of (\ref{4.4}) and (\ref{4.5}) the
 property (iii) holds true with
 \begin{eqnarray}\label{4.6}
 &\vr_{k}=\frac 2{3d}[N^{\frac 14}]^{-\frac 1{24d}}\log([N^{\frac 14}])+
 2\sqrt {C_1}[N^{\frac 14}]^{-\frac 1{24}}\\
 &+4L\sqrt d[N^{\frac 14}]^{-\frac 1{24d}}\leq [N^{\frac 14}]^{-\frac 1{24d}}(\log N
 +2\sqrt {C_1}+4L\sqrt d).\nonumber
 \end{eqnarray}

 Next, we obtain the uniform $L^2$-bound for the difference between the sums of
  $(n_k-n_{k-1})^{1/2}V_k$'s and  of $(n_k-n_{k-1})^{1/2}W_k$'s. Set
 \[
 I(t)=\sum_{0\leq k\leq k_N(t)}(n_k-n_{k-1})^{1/2}(V_k-W_k).
 \]
 \begin{lemma}\label{lem4.2} For any integer $N\geq N_0$,
 \begin{equation}\label{4.7}
 E\max_{0\leq t\leq 1}|I(t)|^2\leq C_2N[N^{\frac 14}]^{-\frac 1{50d}}
 \end{equation}
 where
 \[
 C_2=\sup_{N\geq 1}([N^{\frac 14}]^{-\frac 1{480d}}\sqrt {\log N})(1+4L^2(L^2+ d)+2L^2d)
 (1+\sqrt {2\sqrt {C_1}}+2\sqrt {L\sqrt d} .
 \]
 \end{lemma}
 \begin{proof} Set
 \[
 M_k=\sum_{0\leq l\leq k}(n_l-n_{l-1})^{1/2}(V_l-W_l).
 \]
 Then
 \[
 \max_{0\leq t\leq 1}|I(t)|^2=\max_{1\leq k\leq k_N}|M_k|^2
 \]
 and by the properties (i) and (ii) of Theorem \ref{thm4.1} together with the conditional independence of each $V_l-W_l$ of $\cF_{l-1}\vee\sig\{ U_1,...,U_{l-1}\}$ given $X_N(\frac {n_{l-1}}N)$,
 it is easy to see that $M_k,\, k=1,2,...,k_N$ is a martingale
 with respect to the filtration $\cF_{k}\vee\sig\{ U_1,...,U_k\},\, k=1,...,k_N$. Hence, by the Doob martingale inequality
 \begin{equation}\label{4.8}
 E\max_{1\leq k\leq k_N}|M_k|^2\leq 4E|M_{k_N}|^2=4[N^{\frac 14}]\sum_{1\leq k\leq k_N}E|V_k-W_k|^2
 \end{equation}
 where we use also that $(V_k-W_k)$, $k=1,...,k_N$ are uncorrelated for different $k$'s.

 Next, by the Cauchy-Schwarz inequality
 \begin{eqnarray}\label{4.9}
 &E|V_k-W_k|^{2}=E(|V_k-W_k|^{2}\bbI_{|V_k-W_k|\leq\vr_k})\\
 &+E(|V_k-W_k|^{2}\bbI_{|V_k-W_k|>\vr_k})\nonumber\\
 &\leq\vr^{2}_k+(E|V_k-W_k|^{4})^{1/2}(P\{|V_k-W_k|>\vr_k\})^{\frac 12}\nonumber\\
 &\leq\vr^{2}_k+4\vr^{\frac 12}_k((E|V_k|^{4})^{1/2}+(E|W_k|^{4})^{1/2}).\nonumber
 \end{eqnarray}
 Now, by (\ref{2.2}) and (\ref{2.3}),
 \begin{eqnarray}\label{4.10}
 &E|V|^4\leq [N^{\frac 14}]^{-2}L^4E|\sum_{n_{k-1}<l\leq n_k}\xi(l)|^4\\
 &\leq [N^{\frac 14}]^{-2}L^4\big([N^{\frac 14}]E|\xi(1)|^4+[N^{\frac 14}]^2
 (E|\xi(1)|)^2\big)\leq L^4(L^2+d^2).\nonumber
\end{eqnarray}
 Since $W_k$ is distributed as $\sig(X_N(\frac {n_{k-1}}N))\cN$, where $\cN$ is the $d$-dimensional
 Gaussian random vector with the identity covariance matrix, we obtain that
 \begin{equation}\label{4.11}
 E|W_k|^4\leq 3L^4d^2.
 \end{equation}
 Finally, (\ref{4.7}) follows from (\ref{4.8})--(\ref{4.11}).
 \end{proof}

 Next, let $W(t),\, t\geq 0$ be a $d$-dimensional Brownian motion such that the
increments $W(n_k)-W(n_{k-1})$ are independent of $X_N(\frac {n_{k-1}}N)$ for any $k=1,...,k_N$.
Then, given $X_N(\frac {n_{k-1}}N)$, the sequences of random vectors $\tilde W_k =\sig(X_N(\frac {n_{k-1}}N))(W(n_k)-W(n_{k-1}))$
and $(n_{k}-n_{k-1})^{1/2}W_{k},\, k=1,...,k_N$ have the same distributions. Moreover, we can
redefine the process $\xi(n),\, 1\leq n<\infty$ and choose a Brownian motion $W(s),\, s\geq 0$
preserving their distributions so that the joint distribution of the sequences of pairs $(V_{k},
W_{k})$ and of $(V_{k}, \tilde W_{k})$ will be the same and, in particular, that
 (\ref{4.7}) will hold true with $\tilde W_{k}$ in place of $ W_{k}$. Indeed, by the Kolmogorov extension theorem (see, for instance, \cite{SV}) such pair of processes
 exists if we impose consistent restrictions on their joint finite dimensional distributions. But since
 the pair of processes $\xi$ and $W_{k},\, 1\leq k\leq k_N$ satisfying
 our conditions exist by Theorem \ref{thm4.1} and Lemma \ref{lem4.2}, these restrictions are
 consistent and the required pair of processes exists. From now on we will drop tilde and denote
$\sig(X_N(\frac {n_{k-1}}N))(W(n_{k})-W(n_{k-1}))$ by $W_{k}$ which is supposed to satisfy (\ref{4.7}).

Now, using the Brownian motion $W(t),\, t\geq 0$ constructed above we consider the
new Brownian motion $W_N(t)=N^{-1/2} W(tN),\, 0\leq t\leq 1$ and introduce the diffusion process
 $\Xi_N(t),\, t\geq 0$ solving the stochastic differential equation (\ref{2.1}) which we write now
  with $W_N$,
\[
d\Xi_N(t)=\sig(\Xi_N(t))dW_N(t)+b(\Xi_N(t))dt,\,\, \Xi_N(0)=x_0.
\]
 Now, we introduce the auxiliary process $\hat\Xi_N$ with coefficients frozen at times $n_k$,
\begin{eqnarray*}
&\hat\Xi_N(t)=x_0+\sum_{1\leq k\leq k_N(tN)}\big(\sig(\Xi_N(\frac {n_{k-1}}N))
(W_N(\frac {n_{k}}N)-W_N(\frac {n_{k-1}}N))\\
&+N^{-1}b(\Xi_N(\frac {n_{k-1}}N))(n_{k}-n_{k-1})\big).
\end{eqnarray*}

\begin{lemma}\label{lem4.3} For any integer $N\geq 1$,
\begin{equation}\label{4.12}
E\max_{0\leq k\leq k_N}|\Xi_N(n_k/N)-\hat\Xi_N(n_k/N)|^{2}\leq 32\Del(N)
\end{equation}
where $\Del(N)=N^{-1}[N^{\frac 14}]$.
\end{lemma}
\begin{proof}
First, we write
\begin{eqnarray}\label{4.13}
&E\max_{0\leq k\leq k_N}|\Xi_N(n_k/N)-\hat\Xi_N(n_k/N)|^{2}\\
&\leq 2(E\max_{0\leq k\leq k_N}|J_1(n_k/N)|^{2}+E\max_{0\leq k\leq k_N}|J_2(n_k/N)|^{2})\nonumber
\end{eqnarray}
where
\[
J_1(t)=\int_{0}^{t}\big(\sig(\Xi_N(s))-\sig(\Xi_N([s/\Del(N)]\Del(N)))\big)dW_N(s)
\]
and
\begin{equation*}
J_2(t)=\int_{0}^{t}\big(b(\Xi_N(s))-b(\Xi_N([s/\Del(N)]\Del(N)))\big)ds.
\end{equation*}
 By the Doob martingale inequality and the It\^ o isometry for stochastic integrals
 (see, for instance, \cite{Ki20}, Sections 6.1.2 and 7.2.1),
\begin{eqnarray}\label{4.14}
&E\max_{0\leq k\leq k_N)}|J_1(n_k/N)|^{2}\\
&\leq 4\int_{0}^{[T/\Del(N)]\Del(N)}E|\sig(\Xi_N(s))
-\sig(\Xi_N([s/\Del(N)]\Del(N)))|^{2}ds\nonumber\\
&\leq 4L^{2}\sum_{1\leq k\leq k_N} \int_{n_{k-1}/N}^{n_{k}/N}E|\Xi_N(s)-\Xi_N(n_{k-1}/N)|^{2}ds.
\nonumber\end{eqnarray}

By (\ref{2.2}) and the Cauchy--Schwarz inequality,
\begin{equation}\label{4.15}
E\max_{0\leq k\leq k_N}|J_2(n_k/N)|^{2}\leq L^2\int_{0}^{1}|\Xi_N(s)-\Xi_N([s/\Del(N)]\Del(N))|^2ds.
\end{equation}
Again, by (\ref{2.2}), (\ref{2.3}) and the moment inequalities for stochastic integrals
\begin{eqnarray}\label{4.16}
&E|\Xi_N(s)-\Xi_N(n_{k-1}/N)|^{2}\leq 2\big(E|\int_{n_{k-1}/N}^s\sig(\Xi_N(u))dW_N(u)|^{2}\\
&+L^{2}(s-n_{k-1}/N)^{2}\big)\leq 2L^{2}\Del(N)(1+\Del(N))\leq 4L^2\Del(N)\nonumber
\end{eqnarray}
since $s\in[n_{k-1}/N,n_k/N]$ here, and so $s-n_{k-1}/N\leq \Del(N)$. Now, (\ref{4.12}) follows
 from (\ref{4.13})--(\ref{4.16}).
\end{proof}

 Next, we introduce another auxiliary process $\Xi_N^X$ defined by,
\begin{eqnarray*}
&\Xi_N^X(t)=x_0+\sum_{1\leq k\leq k_N(t)}\big(\sig(X_N(\frac {n_{k-1}}N))
(W_N(\frac {n_{k}}N)-W_N(\frac {n_{k-1}}N))\\
&+N^{-1}b(X_N(\frac {n_{k-1}}N))(n_{k}-n_{k-1})\big).
\end{eqnarray*}
Then we can write
\begin{eqnarray}\label{4.17}
& E\sup_{0\leq s\leq 1}|\hat X_N(s)-\hat\Xi_N(s)|^{2}=E\max_{0\leq k<k_N(TN)}|
\hat X_N(n_k/N)\\
&-\hat\Xi_N(n_k/N)|^{2}\leq 2(E\max_{0\leq k<k_N(TN)}|\hat X_N(n_k/N)-\Xi^X_N(n_k/N)|^{2}
\nonumber\\
&+E\max_{0\leq k<k_N(T)}|\Xi^X_N(n_k/N)-\hat\Xi_N(n_k/N)|^{2})\nonumber.
\end{eqnarray}
By Lemma \ref{lem4.2},
\begin{eqnarray}\label{4.18}
&E\max_{0\leq k\leq n}|\hat X_N(n_k/N)-\Xi^X_N(n_k/N)|^{2}=E\max_{0\leq k\leq n}\\
&\big\vert\sum_{0\leq l\leq k}\sig(X_N(\frac {n_l}N))\big(N^{-\frac 12}\sum_{n_l<m\leq n_{l+1}}\xi(m)-(W_N(\frac {n_{l+1}}N)-W_N(\frac {n_l}N))\big)\big\vert^{2}\nonumber\\
&\leq N^{-1}E\sup_{0\leq t\leq 1}|I(t)|^{2}\leq C_2[N^{\frac 14}]^{-\frac 1{50d}}.\nonumber
\end{eqnarray}

In order to estimate the second term in the right hand side of (\ref{4.17}) introduce the
$\sig$-algebras $\cQ_n=\cF_{0n}\vee\sig\{ W(u),\, 0\leq u\leq n\}$ and observe that
by our construction for each $k$ the increment $W(n_{k+1})-W(n_k)$ is independent of $\cQ_{n_k}$.
 On the other hand, for any $k\geq n$ both $X_N(n_k/N)$ and $\Xi_N(n_k/N)$ are
 $\cQ_{n_k}$-measurable. Hence,
 \begin{equation*}
 \cI_1(n_k)=\sum_{0\leq l\leq k-1}\big(\sig(X_N(\frac {n_{l}}N))-\sig(\Xi_N(\frac {n_{l}}N))\big)
 (W_N(\frac {n_{l+1}}N)-W_N(\frac {n_l}N))
 \end{equation*}
 is a martingale in $k$ with respect to the filtration $\cQ_k, k=1,2,...,k_N-1$. Thus, by (\ref{2.2})
 and the Doob martingale inequality,
 \begin{eqnarray}\label{4.19}
 &E\max_{1\leq k\leq m}|\cI_1(n_k)|^{2}\leq 4 E|\cI_1(n_m)|^{2}\\
 &\leq 4\sum_{0\leq l\leq m-1}E|(\sig(X_N(\frac {n_{l}}N))-\sig(\Xi_N(\frac {n_{l}}N)))
 (W_N(\frac {n_{l+1}}N)-W_N(\frac {n_l}N))|^2\nonumber\\
 &\leq 4dL^2\Del(N)\sum_{0\leq l<m}E|X_N(\frac {n_l}N)-\Xi_N(\frac {n_l}N)|^2.\nonumber
 \end{eqnarray}

Next, observe that
\begin{eqnarray}\label{4.20}
&\max_{0\leq k\leq k_N}|\Xi_N(n_k/N)-\hat\Xi_N(n_k/n)|^{2}\\
&\leq 2(E\max_{0\leq k\leq k_N}|\cI_1(n_k)|^{2}+E\max_{0\leq k\leq k_N}|\cI_2(n_k)|^{2})
\nonumber\end{eqnarray}
where
\begin{eqnarray*}
&\cI_2(n_k)=N^{-1}\sum_{0\leq l\leq k-1}\big(b(X_N(n_{l}))-b(\Xi_N(n_{l}))\big)(n_{l+1}-n_l).
\end{eqnarray*}
By (\ref{2.2}) we have
\begin{eqnarray}\label{4.21}
&|\cI_2(n_k)|^{2}\leq L^2(\Del(N))^2(\sum_{0\leq l<k}|X_N(\frac {n_l}N)-\Xi_N(\frac {n_l}N)|)^2\\
&\leq  L^2(\Del(N))^2k\sum_{0\leq l<k}|X_N(\frac {n_l}N)-\Xi_N(\frac {n_l}N)|^2\nonumber\\
&\leq  L^2\Del(N)\sum_{0\leq l<k}|X_N(\frac {n_l}N)-\Xi_N(\frac {n_l}N)|^2.\nonumber
\end{eqnarray}

Now denote
\[
Q_k=E\max_{0\leq l\leq k}|X_N(n_l/N)-\Xi_N(n_l/N)|^{2}.
\]
Then we obtain from (\ref{3.2}), (\ref{4.12}) and (\ref{4.17})--(\ref{4.21}) that
for $n\leq k_N$,
\begin{equation}\label{4.22}
Q_n\leq C_3[N^{\frac 14}]^{-\frac 1{50d}}+C_4\Del(N)\sum_{0\leq k\leq n-1}Q_k
\end{equation}
where $C_3=408L^8+6C_2+96$ and $C_4=L^2(16d+4)$. By the discrete (time) Gronwall inequality (see, for instance, \cite{Cla}),
\begin{equation}\label{4.23}
Q_{k_N}\leq C_3[N^{\frac 14}]^{-\frac 1{50d}}\exp(C_{4}).
\end{equation}

It remains to estimate deviations of our continuous time processes within intervals of time
$(n_k/N,n_{k+1}/N)$ which where not taken into account in previous estimates,
i.e. we have to deal now with
\begin{eqnarray*}
&\cJ_1=E\sup_{0\leq t\leq 1}|X_N(t)-X_N(n_{k_N(tN)})|^{2}\\
&\mbox{and}\,\,\, \cJ_2=E\sup_{0\leq t\leq 1}|\Xi_N(t)-\Xi_N(n_{k_N(tN)})|^{2}.
\end{eqnarray*}
By the straightforward estimates using (\ref{2.2}) and (\ref{2.4}) we obtain
\begin{equation}\label{4.25}
\cJ_1\leq 2\Del(N)L^2(L^2+1)
\end{equation}
and
\begin{equation}\label{4.26}
\cJ_2\leq 4(\cJ_3+(2L)^{2}(\Del(N))^{2})
\end{equation}
where
\[
\cJ_3=E\max_{0\leq k\leq k_N}\sup_{0\leq s\leq\Del(N)}|\int_{n_k/N}^{N^{-1}n_k+s}
\sig(\Xi_N(u))dW_N(u)|^{2}.
\]

By the Jensen (or Cauchy-Schwarz) inequality and the uniform moment estimates for stochastic
integrals
\begin{eqnarray}\label{4.27}
&\cJ_3\leq\big( E\max_{0\leq k\leq k_N}\sup_{0\leq s\leq\Del(N)}|\int_{n_k/N}^{N^{-1}n_k+s}
\sig(\Xi_N(u))dW_N(u)|^{4}\big)^{1/2}\\
&\leq\big( \sum_{0\leq k\leq k_N}E\sup_{0\leq s\leq\Del(N)}|\int_{n_k/N}^{N^{-1}n_k+s}
\sig(\Xi_N(u))dW_N(u)|^{4}\big)^{1/2}\nonumber\\
&\leq(\frac {4}{3})^{2}\big( \sum_{0\leq k\leq k_N}E|\int_{n_k/N}^{N^{-1}n_{k+1}}
\sig(\Xi_N(u))dW_N(u)|^{4}\big)^{1/2}\nonumber\\
&\leq 6L^2(\Del(N))^{1/2}.\nonumber
\end{eqnarray}
 Combining (\ref{4.23})--(\ref{4.27}) we complete the proof of Theorem \ref{thm2.1}.
 \qed

 \section{Dynkin games }\label{sec5}\setcounter{equation}{0}

In view of the form of our regularity conditions (\ref{2.9}) and (\ref{2.10})  on the payoff
 functionals $F$ and $G$ we will need the following exponential estimates.

 \begin{lemma}\label{lem5.0} (i) For any $M>0$ and an integer $N\geq 1$,
 \begin{equation}\label{5.1.1}
 \max_{0\leq n\leq N}E\exp(M|X_N(n/N)|)\leq D^X_Me^{M|x|}
 \end{equation}
 and
 \begin{equation}\label{5.1.2}
 \max_{0\leq n\leq N}E\exp(M|\hat X_N(n/N)|)\leq D^X_Me^{M|x|}
 \end{equation}
 where $x=X_N(0)=\hat X_N(0)$ and $D^X_M=2d\exp(\frac 12d^4L^4+L+\frac 16M^3d^6L^6e^{Md^2L^2})$
 does not depend on $N$;

 (ii) For any $\del,M>0$ and an integer $N\geq 1$,
 \begin{equation}\label{5.1.3}
 E\exp(M\max_{0\leq n\leq N}|X_N(n/N)|)\leq D^X_Me^{M|x|}N^\del
 \end{equation}
 and
 \begin{equation}\label{5.1.4}
 E\exp(M \max_{0\leq n\leq N}|\hat X_N(n/N)|)\leq D^X_{M,\del}e^{M|x|}N^\del
 \end{equation}
 where $D^X_{M,\del}=1+(D^X_{2M/\del}D^X_{2M})^{1/2}$ also does not depend on $N$;

 (iii) For any $M>0$,
 \begin{equation}\label{5.1.5}
 E\exp(M\sup_{0\leq t\leq 1}|\Xi(t)|)\leq D^\Xi_Me^{M|x|}\,\,\mbox{and}\,\,
 E\exp(M\sup_{0\leq t\leq 1}|\hat \Xi_N(t)|)\leq D^\Xi_Me^{M|x|}
 \end{equation}
 where $x=\Xi(0)$ and $D^\Xi_M=2\exp(L+\frac 12ML^2d^2)$.
 \end{lemma}
\begin{proof}
(i) Writing
\[
X_N(n/N)=x+\sum_{k=0}^{n-1}\big(N^{-1/2}\sig(X_N(k/N))\xi(k+1)+N^{-1}b(X_N(k/N))\big)
\]
we obtain
\begin{eqnarray}\label{5.1.6}
&E\exp(M|X_N(n/N)|)\\
&\leq e^{M(|x|+L)}E\exp(MN^{-1/2}|\sum_{k=0}^{n-1}\sig(X_N(k/N))\xi(k+1)|)\nonumber\\
&\leq E\max_{1\leq i\leq d}\exp(MdN^{-1/2}|\sum_{k=0}^{n-1}\sum_{j=1}^d\sig_{ij}(X_N(k/N))\xi_j(k+1)|)
\nonumber\\
&\leq\sum_{1\leq i\leq d}\big( E\exp(MdN^{-1/2}\sum_{k=0}^{n-1}\sum_{j=1}^d\sig_{ij}(X_N(k/N))\xi_j(k+1))
\nonumber\\
&+ E\exp(-MdN^{-1/2}\sum_{k=0}^{n-1}\sum_{j=1}^d\sig_{ij}(X_N(k/N))\xi_j(k+1))\big).\nonumber
\end{eqnarray}

To shorten a bit notations we set for this proof $g(x)=(g_1(x),...,g_d(x))$ where
$g_j(x)=\pm Md\sig_{ij}(x)$. Then we have to estimate
\begin{eqnarray}\label{5.1.7}
& E\exp\big(N^{-1/2}\sum_{k=0}^{n-1}\langle g(X_N(k/N)),\,\xi(k+1)\rangle\big)\\
&=E\big(\exp(N^{-1/2}\sum_{k=0}^{n-2}\langle g(X_N(k/N)),\,\xi(k+1)\rangle)\nonumber\\
&\times E\big(\exp(N^{-1/2}\langle g(X_N(n-1/N)),\,\xi(n)\rangle)|\cF^\xi_{0,n-1}\big)\big).\nonumber
\end{eqnarray}
Since $|g_j(x)|\leq MdL,\, j=1,...,d$, it follows that
\begin{eqnarray}\label{5.1.8}
&\quad|\exp(N^{-1/2}\langle g(X_N(n-1/N)),\,\xi(n)\rangle)-1-N^{-1/2}\langle g(X_N(n-1/N)),\,\xi(n)\rangle\\
&-\frac 12N^{-1}\langle g(X_N(n-1/N)),\,\xi(n)\rangle^2|\leq\sum_{l=3}^\infty \frac {(Md^2L^2)^l}
{N^{l/2}\l!}\leq \tilde DN^{-3/2}\nonumber
\end{eqnarray}
where $\tilde D=\frac 16M^3d^6L^6e^{Md^2L^2}$. Hence,
\begin{eqnarray}\label{5.1.9}
&E\big(\exp(N^{-1/2}\langle g(X_N(n-1/N)),\,\xi(n)\rangle)|\cF_{0,n-1}\big)\\
&\leq 1+\frac 12N^{-1}\langle g(X_N(n-1/N)),\,\xi(n)\rangle^2\leq 1+\frac 12N^{-1}d^4L^4+
\tilde DN^{-3/2}\nonumber
\end{eqnarray}
where we used that $E(\xi(n)|\cF_{0,n-1})=E\xi(n)=0$. Continuing in the same way with the sums in
the exponent till $n-2,n-3,...,1$ we obtain that
\begin{eqnarray}\label{5.1.10}
& E\exp\big(N^{-1/2}\sum_{k=0}^{n-1}\langle g(X_N(k/N)),\,\xi(k+1)\rangle\big)\\
&\leq (1+\frac 12N^{-1}d^4L^4+N^{-3/2}\tilde D)^N\nonumber\\
&\leq(1+\frac 12N^{-1}d^4L^4)^N(1+\tilde DN^{-3/2})^N
\leq\exp(\tilde D+\frac 12d^4L^4)\nonumber
\end{eqnarray}
proving (\ref{5.1.1}) while (\ref{5.1.2}) follows in the same way.

(ii) Set $\Gam(y)=\{|X_N(n/N)-x|\geq y\}$. By (i) and the exponential Chebyshev inequality
for any $n\leq N,\, y\geq 0$ and $\del>0$,
\[
P\{\Gam(\frac {\del y}{2M})\}\leq D_{\frac {2M}\del}e^{-y}.
\]
Then, taking $y=2\log N$ we have
\begin{eqnarray}\label{5.1.11}
&E\exp(M\max_{0\leq n\leq N}|X_N(n/N)|)\\
&\leq e^{M|x|}E\exp(M\max_{1\leq n\leq N}|X_N(n/N)-x|)\nonumber\\
&\leq e^{M|x|}\big(N^\del+\sum_{n=1}^NE(\bbI_{\Gam_n(\frac \del M\log N)}\exp(M|X_N(n/N)-x|))\big)
\nonumber\\
&\leq e^{M|x|}\big(N^\del+\sum_{n=1}^N\big(P\{\Gam_n(\frac \del M\log N)\}\big)^{1/2}
\big(E\exp(2M|X_N(n/N)-x|))\big)^{1/2}\nonumber\\
&\leq e^{M|x|}\big(N^\del+(D_{2M/\del}D_{2M})^{1/2})\nonumber
\end{eqnarray}
proving (\ref{5.1.3}) while (\ref{5.1.4}) follows in the same way.

For (iii) we have
\begin{eqnarray}\label{5.1.12}
&E\exp(M\sup_{0\leq t\leq 1}|\Xi(t)|)\\
&\leq e^{(M(|x|+L)}E\exp(M\sup_{0\leq t\leq 1}|\int_0^t\sigma(\Xi(s))dW(s)|)\nonumber\\
&\leq e^{(M(|x|+L)}\sum_{i=1}^d\big(E\sup_{0\leq t\leq 1}\exp(Md\sum_{j=1}^d\int_0^t\sigma_{ij}(\Xi(s))dW_j(s))\nonumber\\
&+E\sup_{0\leq t\leq 1}\exp(-Md\sum_{j=1}^d\int_0^t\sigma_{ij}(\Xi(s))dW_j(s))\nonumber.
\end{eqnarray}
Since
\[
\exp\big(\pm Md\sum_{j=1}^d\int_0^t\sigma_{ij}(\Xi(s))dW_j(s)-\frac {M^2d^2}2\int_0^t
\sum_{j=1}^d\sig_{ij}^2(\Xi(s))ds\big)\]
is a martingale with the expectation equal one, it follows from (\ref{2.2}) and the Doob martingale
inequality that
\[
E\sup_{0\leq t\leq 1}\exp(\pm Md\sum_{j=1}^d\int_0^t\sigma_{ij}(\Xi(s))dW_j(s)) \leq
e^{\frac 12M^2L^2d^2},
\]
and so the first inequality in (\ref{5.1.5}) follows while we obtain the second one in the same way.
\end{proof}

Let $\cT^{\Del}$ be the set of all stopping times with respect
to the filtration $\cF^\xi_{0,n_k},\, k\geq 0$ taking on values $n_k,\, k=0,1,...,k_{\max}$
 where $k_{\max}=k_N$ if $k_N=N/[N^{\frac 14}]$ and $k_{\max}=k_N+1$ and $n_{k_{\max}}=N$
if $n_{k_N}<N$. Denote by $\cQ_{n_k}$ the $\sig$-algebra $\cF^\xi_{0,n_k}\vee\sig\{ U_i,\, 1\leq i\leq k\}$ where, recall, $U_1,U_2,...$ is a sequence of i.i.d. uniformly distributed
random variables appearing in Theorem \ref{thm4.1}.  Let $\cT^\cQ$ be the
 set of all stopping times with respect to the filtration $\cQ_{n_k},\, k\geq 0$ taking on values
 $n_k,\, k=0,1,...,k_{\max}$. Next, introduce the payoffs based on $\hat X_N$ (the same as
 in Lemma \ref{lem3.1}),
 \[
 \hat R_N(s,t)=G_s(\hat X_N)\bbI_{s<t}+F_t(\hat X_N)\bbI_{t\leq s}
 \]
 and the Dynkin game values corresponding to the sets of stopping times $\cT^\Del$ and $\cT^\cQ$,
 \[
 V_N^\Del=\inf_{\sig\in\cT^\Del}\sup_{\tau\in\cT^\Del}ER_N(\sig/N,\tau/N),
 \]
 \[
 \hat V_N^\Del=\inf_{\sig\in\cT^\Del}\sup_{\tau\in\cT^\Del}E\hat R_N(\sig/N,\tau/N),
 \]
 \[
 \mbox{and}\,\,\,\hat V_N^\cQ=\inf_{\sig\in\cT^\cQ}\sup_{\tau\in\cT^\cQ}E
 \hat R_N(\sig/N,\tau/N).
 \]

 \begin{lemma}\label{lem5.1} For any $\del>0$ and an integer $N\geq 1$,
 \begin{equation}\label{5.1}
 |V_N-V_N^\Del|\leq D^X_{K,\del}Ke^{K|x|}N^{\del-\frac 14}(1+L+L^2),
 \end{equation}
 where $x=X_N(0)$, and
  \begin{equation}\label{5.2}
 |V_N^\Del-\hat V_N^\Del|\leq 24\sqrt {D^X_{4K,\del}}e^{K|x|}L^4N^{\frac 12(\del-\frac 12)}.
 \end{equation}
 \end{lemma}
 \begin{proof} For any $\zeta\in\cT^\xi_{0N}$ set $\zeta^\Del=\min\{ n_k:\, n_k\geq\zeta\}$
 which defines a stopping time from $\cT^\Del$ satisfying
 \begin{equation}\label{5.3}
 N^{-1}\zeta+\Del(N)\geq N^{-1}\zeta^\Del\geq N^{-1}\zeta.
 \end{equation}
 Since $\cT^\xi_{0N}\supset\cT^\Del$ we see that
 \[
 V_N\geq\inf_{\zeta\in\cT^\xi_{0N}}\sup_{\eta\in\cT^\Del}ER(\zeta/N,\eta/N).
 \]
 Then for any $\vt>0$ there exists $\zeta_\vt\in\cT^\xi_{0N}$ such that
 \[
 V_N\geq\sup_{\eta\in\cT^\Del}ER_N(\zeta_\vt/N,\eta/N)-\vt,
 \]
 and so
 \begin{eqnarray}\label{5.4}
 &V_N\geq\sup_{\eta\in\cT^\Del}ER_N(\zeta_\vt^\Del/N,\eta/N)-\vt\\
 &-\sup_{\eta\in\cT^\Del}E(R_N(\zeta_\vt^\Del/N,\eta/N)-R_N(\zeta_\vt/N,\eta/N))\nonumber\\
 &\geq V_N^\Del-\vt-\sup_{\eta\in\cT^\Del}J_1(\zeta_\vt/N,\eta/N)\nonumber
 \end{eqnarray}
 where for any $\zeta\in\cT^\xi_{0N}$ and $\eta\in\cT^\Del$,
 \[
 J_1(\zeta/N,\eta/N)=E(R_N(\zeta^\Del/N,\eta/N)-R_N(\zeta/N,\eta/N)).
 \]

 Since $\zeta^\Del\geq\zeta$,
 \[
 R_N(\zeta/N,\eta/N)=G_{\zeta/N}(X_N)\,\,\mbox{whenever}\,\, R_N(\zeta^\Del/N,
 \eta/N)=G_{\zeta^\Del/N}(X_N).
 \]
 Hence, by (\ref{2.10}) and (\ref{5.3}),
 \begin{eqnarray}\label{5.5}
& R_N(\zeta^\Del/N,\eta/N)-R_N(\zeta/N,\eta/N)\leq\max\big( |G_{\zeta^\Del/N}(X_N)
-G_{\zeta/N}(X_N)|,\\
& |F_{\zeta^\Del/N}(X_N)-F_{\zeta/N}(X_N)|\big)\nonumber\\
&\leq K\big(\Del(N)(1+L)
+N^{-1/2}\max_{0\leq k\leq k_{\max}}\max_{1\leq l\leq N^{\frac 14}}\nonumber\\
&|\sum_{n_k+l\leq j\leq n_{k+1}}\sig(X_N(j/N)\xi(j)|\big)\exp(K\max_{0\leq n\leq N}|X_N(n/N)|)
\nonumber\\
&\leq K(\Del(N)(1+L)+L^2N^{-1/4})\exp(K\max_{0\leq n\leq N}|X_N(n/N)|).\nonumber
\end{eqnarray}
Taking here $\zeta_\vt$ in place of $\zeta$ we obtain from (\ref{5.4}), (\ref{5.5}) and Lemma
\ref{lem5.0}(ii) that
\[
V_N\geq V_N^\Del-\vt-D^X_{K,\del}Ke^{K|x|}N^{\del-\frac 14}(1+L+L^2)
\]
and since $\vt>0$ is arbitrary we have that
\begin{equation}\label{5.6}
V_N\geq V_N^\Del-D^X_{K,\del}Ke^{K|x|}N^{\del-\frac 14}(1+L+L^2).
\end{equation}

On the other hand, since the Dynkin game here has a value (see, for instance, \cite{Ki20}, Section
6.2.2) we can write also that
\begin{equation}\label{5.7}
V_N=\sup_{\eta\in\cT^\xi_{0N}}\inf_{\zeta\in\cT^\xi_{0N}}ER_N(\zeta/N,\eta/N)\leq\inf_{\zeta\in\cT^\Del}
ER(\zeta/N,\eta_\vt/N)+\vt
\end{equation}
for each $\vt>0$ and some $\eta_\vt\in\cT^\xi_{0N}$. Introducing $\eta_\vt^\Del$ and arguing as
above we obtain that
\[
V_N\leq V_N^\Del+D^X_{K,\del}Ke^{K|x|}N^{\del-\frac 14}(1+L+L^2)
\]
which together with (\ref{5.6}) completes the proof of (\ref{5.1}).

In order to prove (\ref{5.2}) we observe that by (\ref{2.9}), Lemma \ref{lem3.1}, Lemma \ref{lem5.0}(ii), the Chebyshev and the Cauchy-Schwarz inequalities
\begin{eqnarray}\label{5.8}
&|V_N^\Del-\hat V_N^\Del|\leq\sup_{\zeta\in\cT^\Del}\sup_{\eta\in\cT^\Del}E|R_N(\zeta/N,
\eta/N)-\hat R_N(\zeta/N,\eta/N)|\\
&\leq\max\big(E\sup_{0\leq t\leq 1}|F_t(X_N)-F_t(\hat X_N)|,\,|G_t(X_N)-
G_t(\hat X_N)|\big)\nonumber\\
&\leq KE\big(\big(\sup_{0\leq t\leq 1}|X_N(t)-\hat X_N(t)|+\bbI_{\sup_{0\leq t\leq 1}|X_N(t)-\hat X_N(t)|>1}\big)\nonumber\\
&\times\exp(K\sup_{0\leq t\leq 1}(|X_N(t)|+|\hat X_N(t)|))\big)\nonumber\\
&\leq 2K\big(E\sup_{0\leq t\leq 1}|X_N(t)-\hat X_N(t)|^2\big)^{1/2}\nonumber\\
&\times\big( E\exp(4K\sup_{0\leq t\leq 1}|X_N(t)|)\big)^{1/4}\big(E\exp(4K\sup_{0\leq t\leq 1}|\hat X_N(t)|)\big)
^{1/4}\nonumber\\
&\leq 24\sqrt {D^X_{4K,\del}}e^{K|x|}L^4N^{\frac 12(\del-\frac 12)}.\nonumber
\end{eqnarray}
yielding (\ref{5.2}).
  \end{proof}

\begin{lemma}\label{lem5.2} For any integer $N\geq 1$,
\begin{equation}\label{5.9}
\hat V_N^\Del=\hat V_N^\cQ.
\end{equation}
\end{lemma}
\begin{proof}
We prove (\ref{5.9}) obtaining both $\hat V_N^\Del$ and $\hat V_N^\cQ$ by the standard dynamical
programming (backward recursion) procedure (see, for instance, Section 1.3.2 in \cite{Ki20}).
Namely, we have $\hat V_N^\Del=\hat V_{N,0}^\Del$ and $\hat V_N^\cQ=\hat V_{N,0}^\cQ$ where
\begin{equation}\label{5.10}
\hat V_{N,k_{\max}}^\Del=F_T(\hat X)=\hat V_{N,k_{\max}}^\cQ
\end{equation}
proceeding recursively
\[
\hat V_{N,k}^\Del=\min\big(G_{n_k/N}(\hat X_N),\,\max(F_{n_k/N}(\hat X_N),\,
E(\hat V_{N,k+1}^\Del|\cF^\xi_{0,n_k}))\big)
\]
and
\[
\hat V_{N,k}^\cQ=\min\big(G_{n_k/N}(\hat X_N),\,\max(F_{n_k/N}(\hat X_N),\,
E(\hat V_{N,k+1}^\cQ|\cQ_{n_k}))\big).
\]

Since each $\sig$-algebra $\sig\{ U_1,...,U_k\}$ is independent of $\xi_1,\xi_2,...$ by the
construction, i.e. it is independent of all $\sig$-algebras $\cF^\xi_{0,l},\, l=0,\pm 1,...$,
and so it is independent of $X_N$, it follows (see, for instance, \cite{Chu}, p.323
or \cite{Ki07}, Remark 4.3) that
\[
E(\hat V_{N,k+1}^\Del|\cF^\xi_{0,n_k})=E(\hat V_{N,k+1}^\Del|\cQ_{n_k}),
\]
and so starting from (\ref{5.10}) we proceed recursively to $\hat V_{N,0}^\Del=
\hat V_{N,0}^\cQ$
proving (\ref{5.9}).
\end{proof}

Next, we turn our attention to the diffusion $\Xi$ constructed in Theorem \ref{thm2.1} and
consider the corresponding Dynkin game value $V^\Xi$ given by (\ref{2.8}). Set
$\cG^\Xi_{n_k}=\sig\{ W_N(u/N):\, u\leq n_k\}$ and observe that by the construction
\begin{equation}\label{5.11}
\cG^\Xi_{n_k}\subset\cQ_{n_k}=\cF^\xi_{0,n_k}\vee\sig\{ U_i,\, 1\leq i\leq k\}
\end{equation}
where $W_N$ is the Brownian motion constructed in Section \ref{sec4}.
 Let $\cT_\Del^\Xi$ be the set of all stopping times with respect to
the filtration $\cG^\Xi_{n_k},\, k\geq 0$ and $\cT_{\Del}^\cQ$ be the set of all stopping times
with respect to the filtration $\cQ_{n_k},\, k\geq 0$, both taking values $n_k$ when $k$ runs from
0 to $k_{\max}$. Set
\[
\hat R^\Xi(s,t)=G_s(\hat\Xi)\bbI_{s<t}+F_t(\hat\Xi)\bbI_{t\leq s}
\]
where the process $\hat\Xi$ is the same as in Lemma \ref{lem4.3}. Set
 \[
 V^\Xi_\Del=\inf_{\zeta\in\cT^\Xi_\Del}\sup_{\eta\in\cT^\Xi_\Del}ER^\Xi(\zeta/N,\eta/N),
 \]
 \[
 \hat V^\Xi_\Del=\inf_{\zeta\in\cT^\Xi_\Del}\sup_{\eta\in\cT^\Xi_\Del}E
 \hat R^\Xi(\zeta/N,\eta/N)
 \]
 and
 \[
 \hat V^\Xi_\cQ=\inf_{\zeta\in\cT^\cQ}\sup_{\eta\in\cT^\cQ}E\hat R^\Xi(\zeta/N,\eta/N).
 \]

\begin{lemma}\label{lem5.3} For any integer $N\geq 1$,
\begin{equation}\label{5.12}
|V^\Xi-V^\Xi_\Del|\leq K(\Del(N)D_K^\Xi+18L)\sqrt {2D^\Xi_{2K}})
\sqrt {\Del(N)},
\end{equation}
where $x=\Xi(0)$, and
\begin{equation}\label{5.13}
|V^\Xi_\Del-\hat V^\Xi_\Del|\leq (4\sqrt {2}+2L)Ke^{K|x|}\sqrt {D^\Xi_{4K}}\sqrt {\Del(N)}.
\end{equation}
\end{lemma}
\begin{proof}
The proof is similar to Lemma \ref{lem5.1} but here in place of estimates for $X_N$ we have
to use moment estimates for diffusions. Set $\cT_{01}^{\Xi,N}=\{\zeta:\,\zeta/N\in\cT_{01}^\Xi\}$
where, recall, $\cT^\Xi_{01}$ is the set of stopping times with respect to the filtration $\cF_t^\Xi=
\sig\{ W_N(s),\, s\leq t\}$ having values in $[0,1]$. For any $\xi\in\cT_{01}^{\Xi,N}$ define
$\zeta^\Del=\min\{ n_k:\, n_k\geq\zeta\}$ which yields a stopping time from $\cT^\Xi_\Del$ satisfying
(\ref{5.3}). Since $\cT_\Del^\Xi\subset\cT_{0,1}^{\Xi,N}$ we have that
\[
V^\Xi\geq\inf_{\zeta\in\cT_{01}^{\Xi,N}}\sup_{\eta\in\cT_{\Del}^{\Xi}}ER^\Xi(\zeta/N,\eta/N).
\]
In the same way as in (\ref{5.4}) we obtain that for some $\zeta_\vt\in\cT_{01}^{\Xi,N}$,
\begin{equation}\label{5.14}
V^\Xi\geq V^\Xi_\Del-\vt-\sup_{\eta\in\cT_\Del^\Xi}J_2(\zeta_\vt/N,\eta/N)
\end{equation}
where for any $\zeta\in\cT_{01}^{\Xi,N}$ and $\eta\in\cT_{\Del}^\Xi$,
\[
J_2(\zeta/N,\eta/N)=E(R^\Xi(\zeta^\Del/N,\eta/N)-R^\Xi(\zeta/N,\eta/N)).
\]

As in (\ref{5.5}) we obtain from (\ref{2.10}) and (\ref{5.3}) that
\begin{eqnarray}\label{5.15}
&R^\Xi(\zeta^\Del/N,\eta/N)-R^\Xi(\zeta/N,\eta/N)\leq K\big(\Del(N)\\
&+\max_{0\leq k\leq k_{\max}}\sup_{n_k/N\leq s\leq n_{k+1}/N}|\Xi(n_{k+1}/N)-\Xi(s)|\big)\nonumber\\
&\times\exp(K\sup_{0\leq t\leq 1}|\Xi(t)|).\nonumber
\end{eqnarray}
By the Cauchy-Schwarz inequality,
\begin{eqnarray}\label{5.16}
&|E(R^\Xi(\zeta^\Del/N,\eta/N)-R^\Xi(\zeta/N,\eta/N))|\\
&\leq K\Del(N)E\exp(K\sup_{0\leq t\leq 1}|\Xi(t)|)\nonumber\\
&+K\big(E\max_{0\leq k\leq k_{\max}}\sup_{n_k/N\leq s\leq n_{k+1}/N}|\Xi(n_{k+1}/N)-\Xi(s)|^2\big)^{1/2}
\nonumber\\
&\times(E\exp(2K\sup_{0\leq t\leq 1}|\Xi(t)|))^{1/2}.\nonumber
\end{eqnarray}

Next, we write
\begin{eqnarray}\label{5.17}
&\big(E\max_{1\leq k\leq k_{\max}}\sup_{n_{k+1}/N\geq s\geq n_k/N}|\Xi(n_{k+1}/N)-\Xi(s)|^2\big)^{1/2}\\
&\leq(\sum_{1\leq k\leq k_{\max}}E\sup_{n_{k+1}/N\geq s\geq n_k/N}|\Xi(n_{k+1}/N)-\Xi(s)|^4)^{1/4}\nonumber
\end{eqnarray}
and
\begin{eqnarray}\label{5.18}
&E\sup_{n_{k+1}/N\geq s\geq n_k/N}|\Xi(n_{k+1}/N)-\Xi(s)|^4\\
&\leq 8E|\Xi(n_{k+1}/N)-\Xi(n_k/N)|^4+8E\sup_{n_{k+1}/N\geq s\geq n_k/N}|\Xi(s)-\Xi(n_{k}/N)|^4.\nonumber
\end{eqnarray}
By the standard moment estimates for stochastic integrals
\begin{eqnarray}\label{5.19}
&E|\Xi(n_{k+1}/N)-\Xi(n_k)/N|^4\leq 8E|\int_{n_k/N}^{n_{k+1}/N}\sig(\Xi(u))dW_N(u)|^4\\
&+8E(\int_{n_k/N}^{n_{k+1}/N}b(\Xi(u))du)^4\leq 288\Del(N)\int_{n_k/N}^{n_{k+1}/n}E|\sig(\Xi(u))|^4du\nonumber\\
&+8L^4(\Del(N))^4\leq 8L^4(\Del(N))^2(36+(\Del(N))^2)\nonumber
\end{eqnarray}
and
\begin{eqnarray}\label{5.20}
&E\sup_{n_{k+1}/N\geq s\geq n_k/N}|\Xi(s)-\Xi(n_{k}/N)|^4\\
&\leq 8(4/3)^4E|\int_{n_k/N}^{n_{k+1}/N}\sig(\Xi(u))dW_N(u)|^4\nonumber\\
&+8E(\int_{n_k/N}^{n_{k+1}/N}b(\Xi(u))du)^4\leq 8L^4(\Del(N))^2(36(4/3)^4+(\Del(N))^2).\nonumber
\end{eqnarray}

Combining (\ref{5.14})--(\ref{5.20}) together with Lemma \ref{lem5.0}(iii)we obtain the
required lower bound for $V^\Xi-V^\Xi_\Del$
taking into account that $\vt>0$ is arbitrary. On the other hand, since the Dynkin game has a
value under our conditions (see, for instance, \cite{Ki20}, Section 6.2.2) we can write that
\[
V^\Xi=\sup_{\eta\in\cT_{01}^{\Xi,N}}\inf_{\zeta\in\cT_{01}^{\Xi,N}}ER^\Xi(\zeta/N,\eta/N)
\leq\inf_{\zeta\in\cT_{\Del}^{\Xi}}ER^\Xi(\zeta/N,\eta_\vt/N)+\vt
\]
for any $\vt>0$ and some $\eta_\vt\in\cT_{01}^{\Xi,N}$. Introducing $\eta_\vt^\Del$ and relying
on the same arguments as above we obtain the corresponding upper bound for $V^\Xi-V^\Xi_\Del$ and
complete the proof of (\ref{5.12}).

Next, we obtain (\ref{5.13}) by (\ref{2.9}), Lemma \ref{lem4.3}, Lemma \ref{lem5.0}(iii), the
Chebyshev and the Cauchy-Schwarz inequalities,
\begin{eqnarray}\label{5.21}
&|V^\Xi_\Del-\hat V^\Xi_\Del|\leq\sup_{\zeta\in\cT_\Del^\Xi}\sup_{\eta\in\cT_\Del^\Xi}
E|R^\Xi(\zeta/N,\eta/N)-\hat R^\Xi(\zeta/N,\eta/N)|\\
&\leq K\big(E(\max_{0\leq k\leq k_{\max}}|\Xi(k/N)-\hat\Xi(k/N)|+\bbI_{\max_{0\leq k\leq k_{\max}}|\Xi(k/N)-\hat\Xi(k/N)|>1})^2\big)^{1/2}\nonumber\\
&\times\big(E\exp(4K\max_{0\leq k\leq k_{\max}}|\Xi(k/N)|\big)^{1/4}\big(E\exp(4K\max_{0\leq k\leq k_{\max}}|\hat \Xi(k/N)|\big)^{1/4}\nonumber\\
&\leq 2K\sqrt {D^\Xi_{4K}}\big(E\max_{0\leq k\leq k_{N}}|\Xi(k/N)-\hat\Xi(k/N)|^2+E|\Xi(1)-\Xi(k_N/N)|^2\big)^{1/2}\nonumber\\
&\leq (4\sqrt {2}+2L)Ke^{K|x|}\sqrt {D^\Xi_{4K}}\sqrt {\Del(N)}\nonumber
\end{eqnarray}
completing the proof of the lemma.
\end{proof}

Next, we introduce the new process $\Psi_N$, first recursively at the times $N^{-1}n_k$ and then
extending it for all $t\in[0,T]$ in the piece-wise constant fashion. Namely, we set $\Psi_N(0)=x$ and
(with $n_0=0$),
\begin{eqnarray*}
&\Psi_N(N^{-1}n_{k+1})=\Psi_N(N^{-1}n_k)+\sig(\Psi_N(N^{-1}n_{k}))(W_N(N^{-1}n_{k+1})-W_N(N^{-1}n_k))\\
&+N^{-1}b(\Psi_N(N^{-1}n_{k}))(n_{k+1}-n_k)
\end{eqnarray*}
for $k=0,1,...,k_{\max}-1$. Set also $\Psi_N(t)=\Psi_N(N^{-1}n_k)$ if $N^{-1}n_k\leq t<N^{-1}n_{k+1}$.

\begin{lemma}\label{lem5.4} For any integer $N\geq 1$,
\begin{equation}\label{5.22}
E\max_{0\leq k\leq k_{\max}}|\Xi(N^{-1}n_k)-\Psi_N(N^{-1}n_k)|^2\leq 96\Del(N)\exp(24L^2d).
\end{equation}
\end{lemma}
\begin{proof} We have
\begin{eqnarray*}
&|\Xi(N^{-1}n_k)-\Psi_N(N^{-1}n_k)|^2\leq 3\big(|\Xi(N^{-1}n_k)-\hat\Xi(N^{-1}n_k)|^2\\
&+\big\vert\sum_{0\leq l<k}(\sig(\Xi(N^{-1}n_{l}))-\sig(\Psi_N(N^{-1}n_{l})))(W_N(N^{-1}n_{l+1})
-W_N(N^{-1}n_{l}))\big\vert^2\\
&+(N^{-1}\sum_{0\leq l<k}|b(\Xi(N^{-1}n_{l}))-b(\Psi_N(N^{-1}n_{l}))|(n_{l+1}-n_l))^2\big),
\end{eqnarray*}
and so
\begin{eqnarray}\label{5.23}
&\quad\quad\max_{0\leq k\leq n}|\Xi(N^{-1}n_k)-\Psi_N(N^{-1}n_k)|^2\\
&\leq 3\big(\max_{0\leq k\leq n}|\Xi(N^{-1}n_k)-\hat\Xi(N^{-1}n_k)|^2
+\max_{0\leq k\leq n}|M_k|^2\nonumber\\
&+4k_{\max}(\Del(N))^2\sum_{0\leq l<n}|b(\Xi(N^{-1}n_{l}))-
b(\Psi_N(N^{-1}n_{l}))|^2\nonumber
\end{eqnarray}
where
\[
M_k=\sum_{0\leq l<k}(\sig(\Xi(N^{-1}n_{l}))-\sig(\Psi_N(N^{-1}n_{l})))(W_N(N^{-1}n_{l+1})
-W_N(N^{-1}n_{l}))
\]
is a martingale with respect to the filtration $\{ \cG^\Xi_{n_k},\, k\geq 0\}$ since
 $\sig(\Xi(N^{-1}n_{l}))-\sig(\Psi_N(N^{-1}n_{l}))$ is $\cG_{n_{l}}^\Xi$-measurable while
$W_N(N^{-1}n_{l+1})-W_N(N^{-1}n_{l})$ is independent of $\cG_{n_l}^\Xi$.

Hence, by the Doob martingale moment inequality and by the Lipschitz continuity of $\sig$ (with
the constant $L$),
\begin{equation}\label{5.24}
E\max_{0\leq k\leq n}|M_k|^2\leq 4E|M_n|^2\leq 4L^2dN^{-1}\sum_{0\leq k< n}Q_k(n_{k+1}-n_k)
\end{equation}
where
\[
Q_n=E\max_{0\leq k\leq n}|\Xi(N^{-1}n_k)-\Psi_N(N^{-1}n_k)|^2.
\]
By (\ref{5.23}), (\ref{5.24}) and Lemma \ref{lem4.3} we obtain that
\[
Q_n\leq 96\Del(N)+24L^2d\Del(N)\sum_{0\leq k<n}Q_k.
\]
Thus, by the discrete (time) Gronwall inequality (see \cite{Cla}),
\[
Q_n\leq 96\Del(N)\exp(24L^2d\Del(N)n)
\]
and since $n\leq k_{\max}$, (\ref{5.22}) follows.
\end{proof}

Next, we introduce the values of Dynkin games with payoffs based on the process $\Psi_N$. Namely,
we set
\[
R^\Psi_N(s,t)=G_s(\Psi_N)\bbI_{s<t}+F_t(\Psi_N)\bbI_{t\leq s},
\]
 \[
 V^\Psi_\Del=\inf_{\zeta\in\cT^\Xi_\Del}\sup_{\eta\in\cT^\Xi_\Del}ER^\Psi_N(N^{-1}\zeta,N^{-1}\eta)
 \]
 \[
\mbox{and}\,\, V^\Psi_\cQ=\inf_{\zeta\in\cT^\cQ}\sup_{\eta\in\cT^\cQ}
ER^\Psi_N(N^{-1}\zeta,N^{-1}\eta).
 \]
\begin{lemma}\label{lem5.5} For any $\ve>0$,
\begin{equation}\label{5.25}
V^\Psi_\Del=V^\Psi_\cQ.
\end{equation}
\end{lemma}
\begin{proof}
As in Lemma \ref{lem5.2} we will prove (\ref{5.25}) obtaining both $V_\Del^\Psi$ and $V^\Psi_\cQ$ by
the dynamical programming procedure. Again, we have $V^\Psi_\Del=V^\Psi_{\Del,0}$ and   $V^\Psi_\cQ=
V^\Psi_{\cQ,0}$ where $V^\Psi_{\Del,k_{\max}}=F_T(\Psi_N)=V^\Psi_{\cQ,k_{\max}}$ and for
$k=k_{\max}-1, k_{\max}-2,...,0$,
\[
V^\Psi_{\Del,k}=\min\big(G_{N^{-1}n_k}(\Psi_N),\,\max(F_{N^{-1}n_k}(\Psi_N),\, E(V^\Psi_{\Del,k+1}|
\cG_{n_k}^\Xi))\big)
\]
and
\[
V^\Psi_{\cQ,k}=\min\big(G_{N^{-1}n_k}(\Psi_N),\,\max(F_{N^{-1}n_k}(\Psi_N),\, E(V^\Psi_{\cQ,k+1}|
\cQ_{n_k}))\big).
\]

For any vectors $x_0,x_1,x_2,...,x_{k_{\max}}\in\bbR^d$ set $x(0)=x_0$, $x(t)=x_k$ if $N^{-1}n_k\leq t
<N^{-1}n_{k+1}$ and define the functions
\[
q_{k_N(t)}(x_1,...,x_{k_N(t)})=F_t(x)\,\,\mbox{and}\,\,
r_{k_N(t)}(x_1,...,x_{k_N(t)})=G_t(x).
\]
Introduce
\[
\Phi_l(x_1,...,x_l)=\min\big(r_l(x_1,...,x_l),\,\max(q_l(x_1,...,x_l),\, h(x_1,...,x_l))\big)
\]
where
\[
h(x_1,...,x_l)=E\Phi_{l+1}\big(x_1,...,x_l,\, x_l+\sig(x_{l})(W_N(N^{-1}n_{l+1})-W_N(N^{-1}n_l))\big).
\]
Since $\Psi_N(N^{-1}n_l)$ is both $\cG_{n_l}$ and $\cQ_{n_l}$-measurable while
$W_N(N^{-1}n_{l+1})-W_N(N^{-1}n_l)$ is idependent of both $\cG_{n_l}$ and $\cQ_{n_l}$ we see by
induction that
\[
V^\Psi_{\cQ,l}=\Phi_l(\Psi_N(N^{-1}n_1),\Psi_N(N^{-1}n_2),...,\Psi_N(N^{-1}n_l))=V^\Psi_{\Del,l},
\]
for all $l=k_{\max},k_{\max}-1,...,0$ where $\Phi_0=\min(F_0(x_0),\max(G_0(x_0),E\Phi_1(x_0+
\sig(x_0)W_N(N^{-1}n_1)))$,
and (\ref{5.25}) follows.
\end{proof}

Now we can complete the proof of Theorem \ref{thm2.2} writing first,
\begin{eqnarray}\label{5.26}
&|V^\Xi-V_N|\leq |V_N-V_N^\Del|+|V_N^\Del-\hat V_N^\cQ|+|\hat V_N^\cQ-V^\Psi_\cQ|\\
&+|V^\Psi_\cQ-\hat V^\Xi_\Del|+|\hat V^\Xi_\Del-V^\Xi_\Del| +|V^\Xi_\Del-V^\Xi|.\nonumber
\end{eqnarray}
 It remains to estimate $|\hat V_N^\cQ-V^\Psi_\cQ|$ and $|V^\Psi_\cQ-\hat V^\Xi_\Del|=|V^\Psi_\Del-\hat V^\Xi_\Del|$ since all other terms in the right hand side of (\ref{5.26}) are dealt with by Lemmas
 \ref{lem5.1}--\ref{lem5.3}. In both remaining estimates we use the fact that the game values there are
 defined with respect to the same sets of stopping times which will allow us to rely on uniform bounds
 on distances between the corresponding processes. By (\ref{2.9}) and the Cauchy-Schwarz inequality,
 \begin{eqnarray}\label{5.27}
 &\quad|\hat V_N^\cQ-V^\Psi_\cQ|\leq\sup_{\zeta\in\cT^{\cQ}}\sup_{\eta\in\cT^{\cQ}}E|\hat R(N^{-1}\zeta,N^{-1}\eta)- R^\Psi_N(N^{-1}\zeta,N^{-1}\eta)|\\
 &\leq\max(E\sup_{0\leq t\leq 1}|F_t(\hat X_N)-F_t(\Psi_N)|,\,
 E\sup_{0\leq t\leq 1}|G_t(\hat X_N)-F_t(\Psi_N)|)\nonumber\\
 &\leq \sqrt 2K(E\max_{0\leq k\leq k_{\max}}|\hat X_N(N^{-1}n_k)-\Psi_N(N^{-1}n_k)|^2\nonumber\\
 &+P\{\max_{0\leq k\leq k_{\max}}|\hat X_N(N^{-1}n_k)-\Psi_N(N^{-1}n_k)|>1\}\big)^{1/2}\nonumber\\
 &\times\big(E\exp(2K(\max_{0\leq k\leq k_{\max}}(|\hat X_N(N^{-1}n_k)|+
 |\Psi_N(N^{-1}n_k)|)))\big)^{1/2}.\nonumber
 \end{eqnarray}
 Next, by Lemmas \ref{lem3.1}, \ref{lem5.4} and Theorem \ref{thm2.1},
  \begin{eqnarray}\label{5.28}
&E\max_{0\leq k\leq k_{\max}}|\hat X_N(N^{-1}n_k)-\Psi_N(N^{-1}n_k)|^2\\
&\leq  3E\max_{0\leq k\leq k_{\max}}|\hat X_N(N^{-1}n_k)-X_N(N^{-1}n_k)|^2\nonumber\\
&+3E\max_{0\leq k\leq k_{\max}}|X_N(N^{-1}n_k)-\Xi(N^{-1}n_k)|^2\nonumber\\
&+3E\max_{0\leq k\leq k_{\max}}|\Xi(N^{-1}n_k)-\Psi_N(N^{-1}n_k)|^2\nonumber\\
&\leq 408L^8N^{-1/2}+3C_0[N^{\frac 14}]^{-\frac 1{50d}}+96\exp(24L^2d)
\Del(N). \nonumber
\end{eqnarray}
In view of the Chebyshev inequality the probability in (\ref{5.27}) is also estimated by
the right hand side of (\ref{5.28}).

Similarly, by (\ref{2.9}) and by Lemmas \ref{lem4.3}, \ref{lem5.0} and \ref{lem5.4},
\begin{eqnarray}\label{5.29}
&|V^\Psi_\Del-\hat V^\Xi_\Del|\leq\sup_{\zeta\in\cT^{\Del}}\sup_{\eta\in\cT^{\Del}}E|R^\Psi_N(N^{-1}\zeta,N^{-1}\eta)\\
&- \hat R^\Xi(N^{-1}\zeta,N^{-1}\eta)|\leq 2K \big(E\max_{0\leq k\leq k_{\max}}|\Psi_N(N^{-1}n_k)
-\hat\Xi(N^{-1}n_k)|^2\big)^{1/2}\nonumber\\
&\times\big(E\exp(2K(\max_{0\leq k\leq k_{\max}}(|\Psi_N(N^{-1}n_k)|+
|\hat\Xi(N^{-1}n_k)|)))\big)^{1/2}\nonumber\\
&\leq 2\sqrt 2K\big (E\max_{0\leq k\leq k_{\max}}|\Psi_N(N^{-1}n_k)-\Xi(N^{-1}n_k)|^2\nonumber\\
&+E\max_{0\leq k\leq k_{\max}}|\Xi(N^{-1}n_k)-\hat\Xi(N^{-1}n_k)|^2\big)^{1/2}(D^\Xi_{4K})^{1/2}
\nonumber\\
&\leq 16K\sqrt {\Del(N)}(1+3\exp(24L^2d))^{1/2}(D^\Xi_{4K})^{1/2}.\nonumber
\end{eqnarray}
Combining (\ref{5.26}) together with (\ref{5.27})--(\ref{5.29}) and Lemmas \ref{lem5.0}--\ref{lem5.3}
we complete the proof of Theorem \ref{thm2.2}.
\qed


\end{document}